\documentclass[letterpaper, 10 pt, conference]{ieeeconf}  

\IEEEoverridecommandlockouts                              
\usepackage[backend=bibtex,
style=ieee,
bibencoding=ascii,
]{biblatex}
\usepackage{amsfonts}
\usepackage{mathrsfs}

\usepackage{amsthm}
\newtheoremstyle{mystyle}
  {}
  {}
  {\itshape}
  {}
  {\bfseries}
  {.}
  { }
  {}

\theoremstyle{mystyle}

\newtheorem{theorem}{Theorem}[section]

\newtheorem{proposition}[theorem]{Proposition}
\newtheorem{corollary}[theorem]{Corollary}

\newtheorem{example}[theorem]{Example}

\newtheorem{remark}[theorem]{Remark} 
 
\usepackage{enumerate}

\usepackage{subcaption} 

\usepackage{amsmath,amssymb,amsfonts}
\usepackage{algorithmic}
\usepackage{graphicx}

\title{\textbf{On continuation and convex Lyapunov functions}}
\author{Wouter Jongeneel and Roland Schwan
\thanks{\today. Both authors are supported by the Swiss National Science Foundation under the NCCR \textit{Automation}, grant agreement~51NF40\_180545.}
\thanks{Wouter Jongeneel is with the Risk Analytics and Optimization Chair (EPFL). Roland Schwan is with the Automatic Control Lab (EPFL) and the Risk Analytics and Optimization Chair (EPFL). Contact: \text{\{wouter.jongeneel, roland.schwan\}}\text{@epfl.ch}. The authors are grateful to Matthew D. Kvalheim and Matteo Tacchi for their feedback.}}

\begin{document}

\maketitle
\thispagestyle{plain}
\pagestyle{plain}

\begin{abstract}
Suppose that the origin is globally asymptotically stable under a set of continuous vector fields on Euclidean space and suppose that all those vector fields come equipped with---possibly different---convex Lyapunov functions. We show that this implies there is a homotopy between any two of those vector fields such that the origin remains globally asymptotically stable along the homotopy. Relaxing the assumption on the origin to any compact convex set or relaxing convexity to geodesic convexity does not alter the conclusion. Imposing the same convexity assumptions on control Lyapunov functions leads to a Hautus-like stabilizability test. These results ought to be of interest in the context of learning stability certificates, policy gradient methods and switched systems.
\end{abstract}

\section{Introduction}
\label{sec:introduction}
{
Ever since the time of Descartes, \textit{convexity} has been recognized as an important notion to study structural properties of various mathematical objects~\cite{ref:fenchel1983convexity}. In this work, we aim to improve our understanding of the topology of spaces of stable systems and show that again convexity plays an important role. In particular, we study if vector fields on $\mathbb{R}^n$ with a common global attractor can be \textit{continuously transformed} (formally, homotoped, see Section~\ref{sec:cvx}) into each other while preserving the attractor along the transformation. Before making this precise, we briefly elaborate on the relevance of this question.

One can argue that, originally, this question emerged in the dynamical systems community. That is, well over 40 years ago, Conley asked if dynamical systems with qualitatively similar properties can be continuously transformed into each other while preserving those properties along the transformation~\cite[p.~83]{ref:conley1978isolated}. In this work we address---in arguably the most simple setting---those \textit{continuation} (see Section~\ref{sec:continuation}) questions as posed by Conley~\cite{ref:conley1978isolated} and later Kvalheim~\cite{ref:kvalheim2022obstructions}. Our setting is simple in the sense that we largely focus on stability of equilibrium points instead of general attractors and spaces. 

Then, in the context of linear optimal control, policy gradient methods have been recently shown to be a powerful controller synthesis paradigm as data and constraints can be naturally incorporated~\cite{ref:hu2023toward}. Omitting details, these algorithms are frequently studied as being discretizations of a continuous-time gradient flow~\cite{ref:bu2019lqr}. Here, the common assumption is that the optimal control cost is only finite under a stabilizing controller. Now, the hope is that, if initialized properly, gradient flow gives rise to a curve of \textit{stabilized} closed-loop systems, moving from some initial closed-loop system to an optimal closed-loop system. As such, it is of great importance to understand \textit{a priori} when such a curve exists, especially when moving beyond linear systems. 
For instance, when such a curve does not exist, a step size cannot be made arbitrarily small~\textit{cf.}~\cite[Thm.~1]{ref:hu2023toward} as one might need to ``\textit{jump}'', similarly, initialization becomes of critical importance. Indeed, the importance of understanding the topology of the space of stable systems has been recognized early on,~\textit{e.g.}, see~\cite{Brockett1,ober1987topology} and motivated by these gradient-based methods this question received renewed interest,~\textit{e.g.}, see~\cite{ref:feng2019exponential,ref:bu2019lqr,ref:bu2020topological}.  

A more surprising motivating example can be found in the context of switched systems. It turns out that if we have two vector fields such that the origin is globally asymptotically stable (GAS) (see Section~\ref{sec:stability}), then, the origin remains GAS under arbitrary switching between those two vector fields only if those two vector fields can be continuously transformed into each other such that along the transformation the origin remains GAS (see Proposition~\ref{prop:nec:switched} below).   

Motivated by the above, this work aims to illustrate how the path-connectedness of spaces of dynamical systems can be studied via structural properties of Lyapunov functions. In particular, motivated by recent advances in learning~\cite{ref:amos2017input,ref:kolter2019learning}, we focus on the ramifications of assuming (control) Lyapunov functions---as pioneered by Artstein~\cite{ref:artstein1983stabilization} and Sontag~\cite{ref:sontag1989universal}---to be convex. Overall, this work is also in the spirit of the work by Arnold~\cite[Sec.~22]{arnold1973ordinary}, Zabczyk~\cite{ref:zabczyk1989}, Reineck~\cite{ref:reineck1991continuation},~Sepulchre \& Aeyels~\cite{ref:sepulchre1996homogeneous}, Gr\"une, Sontag \& Wirth~\cite{grune1999asymptotic}, Coron~\cite[Ch.~11]{ref:coron2007control}, Byrnes~\cite{ref:byrnes2008brockett} and Cieliebak \& Eliashberg~\cite[Ch.~9]{ref:cieliebak2012stein}.
}

We start by introducing Lyapunov functions for the dynamical control systems at hand. Then, in Section~\ref{sec:topo} we highlight topological properties of level sets of Lyapunov functions. These observations are the motivation for Section~\ref{sec:cvx}-\ref{sec:continuation} where we infer continuation results by considering several notions of convexity. This article is concluded in Section~\ref{sec:conclusion}. 

\noindent\textit{\textbf{Notation:}} Let $r\in \mathbb{N}\cup \{\infty\}$, then, $C^r(U;V)$ denotes the set of $C^r$-smooth functions from $U$ to $V$. The inner product on $\mathbb{R}^n$ is denoted by $\langle \cdot , \cdot \rangle$ and $\mathbb{S}^{n-1}=\{x\in \mathbb{R}^n:\|x\|_2=1\}$. The Lie derivative of a smooth function $h$ over some open set $U\subseteq \mathbb{R}^n$ with respect to a smooth vector field $X$ over $U$ is denoted by $L_X h$ and is defined pointwise by $L_Xh(p):=\langle \nabla h(p), X(p)\rangle$ for any $p\in U$~\cite[Prop.~12.32]{Lee2}. By $\mathrm{cl}(W)$ we denote the closure of $W$ and by $\mathrm{int}(W)$ we denote its interior. The map $x\mapsto x$ on $\mathbb{R}^n$ is denoted by $\mathrm{id}_{\mathbb{R}^n}$ and tangent spaces of appropriate sets $M$ are denoted by $T_pM$, for $p\in M$, with $TM$ denoting the corresponding tangent bundle~\cite[p.~65]{Lee2}.      

\subsection{Dynamical control systems}
We study \textit{\textbf{dynamical systems}} over $\mathbb{R}^n$ of the form
\begin{equation}
    \label{equ:dyn:sys}
    \frac{\mathrm{d}}{\mathrm{d}t}x(t)=F(x(t)) : \begin{cases}
        F:\mathbb{R}^n\to T\mathbb{R}^n\\
        \pi\circ F = \mathrm{id}_{\mathbb{R}^n},
    \end{cases}
\end{equation}
where $F$ is $C^r$-smooth with $r\geq 0$, $\pi:T\mathbb{R}^n\to \mathbb{R}^n$ defined by $(x,v)\mapsto \pi(x,v)=x$ is the canonical projection and for any $x\in \mathbb{R}^n$ we have with some abuse of notation $F(x)\in T_x\mathbb{R}^n$. Evidently, $T\mathbb{R}^n\simeq \mathbb{R}^n\times \mathbb{R}^n$, but~\eqref{equ:dyn:sys} is useful to keep in mind when comparing objects to assess if generalizations beyond $\mathbb{R}^n$ are possible. \textit{Integral curves} of~\eqref{equ:dyn:sys} are differentiable curves $t\mapsto \xi(t)\in \mathbb{R}^n$ such that $\dot{\xi}(t)=F(\xi(t))$ for all $t\in \mathrm{dom}(\xi)$, which is non-empty by, for instance, assuming that $r\geq 1$. However, in general, such an assumption is too strong. We will not go into further regularity conditions and always assume, for simplicity, that $r=0$ and that the vector field is \textit{complete},~\textit{i.e.}, a global flow (see below) is induced, such that we are allowed to make global statements\footnote{We remark that \textit{completeness} is the important property here as we will appeal to a global flow, \textit{smoothness} of $F$ (going beyond $C^0$), on the other hand, is rarely exploited. The only reason to potentially keep smoothness is that one can naturally relax completeness and make some local statements. Without completeness, global statements can break down, consider $\dot{x}=x^2$. However, as the emphasis of this article is on \textit{global asymptotic stability}, examples of that form are somewhat obsolete.}, for further information we point the reader to~\cite{ref:sontag2013mathematical,ref:hale2009ordinary}.

Going beyond \textit{descriptions}, when aiming to \textit{prescribe} the dynamics of a system we consider (time-invariant) \textit{\textbf{dynamical control systems}} over $\mathbb{R}^n\times \mathbb{R}^m$ of the form 
\begin{equation}
\label{equ:control:sys}
    \frac{\mathrm{d}}{\mathrm{d}t}x(t) = f(x(t),u)
\end{equation}
such that $f(x,u)\in T_x\mathbb{R}^n\simeq \mathbb{R}^n$ for all $(x,u)\in \mathbb{R}^n\times \mathbb{R}^m$, where $x$ and $u$ denote the state and input, respectively. 
Again, with some abuse of notation, we will assume that $f\in C^{0}(\mathbb{R}^n\times \mathbb{R}^m;\mathbb{R}^n)$, but again omit integrability discussions.
Input functions are of the form $t\mapsto \mu(t)\in \mathbb{R}^m$,~\textit{e.g.}, a state feedback is of the form $t\mapsto\mu(x(t))$.
Note, we use $\mu$ instead of $u$ to differentiate between the function and the point. A subclass of~\eqref{equ:control:sys} of interest are the so-called \textit{control affine} systems of the form
\begin{equation}
    \label{equ:control:sys:affine}
    \frac{\mathrm{d}}{\mathrm{d}t}x(t) = f(x(t)) + \textstyle\sum^m_{i=1}g_i(x(t))u_i,
\end{equation}
where $u_i$ is the $i^{\mathrm{th}}$ element of $u\in \mathbb{R}^m$ and again $f,g_i\in C^{0}(\mathbb{R}^n;\mathbb{R}^n)$ for $i=1,\dots,m$~\cite{nonlin}. Based on the control system at hand, one might say more about the space of allowable inputs $t\mapsto \mu(t)$,~\textit{e.g.}, one might consider \textit{absolutely integrable} ($L^1_{\mathrm{loc}}$) or \textit{essentially bounded} ($L^{\infty}_{\mathrm{loc}}$) function spaces~\cite[App.~C]{ref:sontag2013mathematical}.

\subsection{Stability}
\label{sec:stability}
Let $F$ parametrize a dynamical system of the form~\eqref{equ:dyn:sys}. By our standing completeness and smoothness assumptions, $F$ will give rise to a continuous \textit{flow}\footnote{Flows satisfy: (1) the \textit{identity} $\varphi^0=\mathrm{id}_{\mathbb{R}^n}$; and (2) \textit{group} property $\varphi^{s+t}=\varphi^s\circ\varphi^t$ $\forall s,t\in \mathbb{R}$.} $\varphi:\mathbb{R}\times \mathbb{R}^n\to \mathbb{R}^n$, with its evaluation denoted by $\varphi^t(x_0):=\varphi(t,x_0)$, which is understood to describe a solution to~\eqref{equ:dyn:sys} at time $t$, starting at time $0$ from $x_0$. A point $x^{\star}\in \mathbb{R}^n$ is an \textit{equilibrium point} of $F$ when $F(x^{\star})=0$, \textit{w.l.o.g.} we set $x^{\star}=0$. Then, $0$ is said to be \textit{\textbf{globally asymptotically stable}} (GAS) (with respect to $F$) if 
\begin{enumerate}[(s-i)]
    \item $0$ is \textit{Lyapunov stable}, that is, for any open neighbourhood $U_{\varepsilon}\ni 0$ there is an open set $U_{\delta}\subseteq U_{\varepsilon}$ such that a solution (with respect to $F$) starting in $U_{\delta}$ stays in $U_{\varepsilon}$;
    \item $0$ is \textit{globally attractive}, that is, $\lim_{t\to+\infty}\varphi^t(x_0)=0$ for all $x_0\in \mathbb{R}^n$.
\end{enumerate}
We will not further digress into solutions and stability and refer to~\cite{ref:sontag2013mathematical}. In general it is not straightforward to capture if $0$ is GAS or not. A fruitful tool that \textit{does} allow for conclusions of this form has been devised by Lyapunov in the late 1800s~\cite{ref:liapunov1892general}. A function $V\in C^{\infty}(\mathbb{R}^n;\mathbb{R}_{\geq 0})$ is said to be a (smooth, strict and proper) \textbf{\textit{Lyapunov function}} (with respect to $F$ and $0$) when
\begin{enumerate}[(V-i)]
    \item \label{prop:i:V}$V(x)>0$ for all $x\in \mathbb{R}^n\setminus\{0\}$ and $V(0)=0$;
    \item \label{prop:ii:V}$\langle \nabla V(x),F(x) \rangle <0$ for all $x\in \mathbb{R}^n\setminus\{0\}$;
    \item \label{prop:iii:V}and $V$ is \textit{radially unbounded}, that is, $V(x)\to +\infty$ for $\|x\|\to+\infty$.
\end{enumerate}
Property~(V-\ref{prop:iii:V}) implies sub-level set compactness. 
Now, based on work by Massera, Kurzweil and others~\cite{ref:kurzweil1963inversion,ref:fathi2019smoothing}, we will exploit the celebrated theorem stating that $0$ is GAS if and \textit{only if} there is a (corresponding) smooth Lyapunov function~\cite[Thm.~2.4]{ref:bacciotti2005liapunov}. Note, we dropped the adjective ``\textit{strict and proper}'' as we exclusively look at Lyapunov functions of that form. See also that, given that $V$ satisfies Property~(V-\ref{prop:i:V}), then $\langle \nabla V(x),F(x)\rangle \leq -V(x)$ implies Property~(V-\ref{prop:ii:V}).  

For further references on Lyapunov stability theory we point the reader to~\cite{ref:bhatia1970stability,ref:sontag2013mathematical,ref:bacciotti2005liapunov}. 

Now, given a control system~\eqref{equ:control:sys}, when it comes to the task of \textit{globally asymptotically stabilizing} $0$ (we will exclusively focus on stabilization by means of state feedback\footnote{Considering more general input functions,~\textit{e.g.}, of the form $t\mapsto \mu(t,x(t))$, integral curves of the corresponding closed-loop system are generally understood to be absolutely continuous curves $\xi:I\to \mathbb{R}^n$ such that the differential relation $\dot{\xi}(t)=F(\xi(t),\mu(t,\xi(t)))=:F'(t,\xi(t))$ holds for almost all $t\in I$, in the sense of Lebesgue. This requires rethinking some concepts,~\textit{e.g.}, global asymptotic stability and what a closed-loop \textit{vector field} really is.}), the Lyapunov function paradigm can be adjusted. Given our stabilization goal, we seek a function $t\mapsto \mu(x(t))$ such that under $f(x,\mu(x))=:F(x)$ the origin is GAS. Then, analogously to the definition of a Lyapunov function, one can define \textit{\textbf{control Lyapunov functions}} (CLFs), yet, Property~(V-\ref{prop:ii:V}) is now replaced by asking that for any $x\in \mathbb{R}^n\setminus\{0\}$ the following holds
\begin{equation}
\label{equ:CLF:cond}
    \inf_{u\in \mathbb{R}^m}\langle \nabla V(x), f(x,u) \rangle < 0. 
\end{equation}
It is not evident that a choice of input function based on~\eqref{equ:CLF:cond} can result in a continuous---let alone smooth---feedback. The next section elaborates on this problem.

\subsection{On control Lyapunov functions}
Consider a dynamical control affine system with scalar input of the form
\begin{equation}
\label{equ:basic:control:sys}
    \frac{\mathrm{d}}{\mathrm{d}t}x(t) = f(x(t))+g(x(t))u,
\end{equation}
Then, for $V$ to be a smooth CLF for~\eqref{equ:basic:control:sys}, we must have that for any $x\in \mathbb{R}^n\setminus\{0\}$ there exists a $u\in \mathbb{R}$ such that $L_f V(x) + u L_g V(x) < 0$. However, the existence of a \textit{smooth} \textit{control}-Lyapunov function is topologically strong in the sense that it generally implies (see below) that an asymptotically stabilizing \textit{continuous} feedback exists~\cite[Ch.~5]{ref:sontag2013mathematical}. Indeed, the controller attributed to Sontag is 
\begin{equation}
    \label{equ:Sontag}
    \mu_s(x) := \begin{cases}
  - \dfrac{\alpha(x)+\sqrt{\alpha(x)^2+\beta(x)^4}}{\beta(x)} \, & \text{if } \beta(x)\neq 0\\
  0 \, & \text{otherwise}, 
    \end{cases}
\end{equation}
for $\alpha(x):=L_f V(x)$ and $\beta(x):=L_gV(x)$,
\textit{e.g.}, see~\cite[p.~249]{ref:sontag2013mathematical}. Although~\eqref{equ:Sontag} appears singular, $\mu_s(x)$ can be shown to be continuous under the following condition; we speak of the \textit{{small control property}} when for all $\varepsilon>0$ there is a $\delta>0$ such that if $x\in \mathbb{R}^n\setminus \{0\}$ satisfies $\|x\|<\delta$, then, there is a $u$ such that $\|u\|<\varepsilon$ and $L_fV(x)+L_gV(x) u<0$~\cite[p.~247]{ref:sontag1989universal}. As such, the existence of a \textit{smooth} CLF is strictly stronger than being (globally) asymptotically controllable\footnote{See for example~\cite[Sec.~2]{ref:rifford2002semiconcave} and references therein for more on this notion.},~\textit{e.g.}, continuous feedback can be easily obstructed for globally controllable systems that even admit smooth CLFs\footnote{A well-known example attributed to Ledyaev \& Sontag is of the form $\dot{x}_1=u_2u_3$, $\dot{x}_2=u_1u_3$, $\dot{x}_3=u_1u_2$~\textit{cf.}~\cite{ref:ledyaev1999lyapunov}.}. Hence, the small control property does not always hold and it is well-known that CLF-based-controllers can be singular, and ever since their inception so-called ``\textit{desingularization techniques}'' emerged~\cite[Sec.~12.5.1]{ref:coron2007control}. For instance, under structural assumptions a backstepping approach to handle CLF singularities is studied in~\cite{ref:li1997maximizing} and a PDE reformulation to avoid singularities is presented in~\cite{ref:yamashita2000global}.

Nevertheless, in case the dynamical control system is affine in the input $u$, and $u$ is constrained to a compact convex set, then, the \textit{existence} of a $C^{\infty}$ CLF is equivalent to the existence of a $C^0$ (on $\mathbb{R}^n\setminus\{0\}$) stabilizing feedback~\cite{ref:artstein1983stabilization}. Indeed, the work by Sontag aimed at making the \textit{construction} of such a feedback transparent. Further relaxing regularity of a CLF, it can be shown that the existence of a so-called ``\textit{proximal CLF}'' is equivalent to asymptotic controllability. These proximal CLFs are $C^r$-smooth with $r\in [0,1)$,~\textit{e.g.}, see~\cite{ref:clarke2010discontinuous} for more on non-smooth CLFs. Better yet, it can be shown that global asymptotic controllability implies the existence of a---possibly discontinuous---feedback~\cite{clarke1997asymptotic}. Even more, Rifford showed that when the control system is globally asymptotically controllable, a---possibly nonsmooth---semiconcave\footnote{A continuous function $f$ is said to be \textit{semiconcave} when there is a $C>0$ such that $x\mapsto f(x)-C\|x\|_2^2$ is concave.} CLF always exists. Exploiting this structure, for control affine systems, Rifford could extend Sontag's formula~\eqref{equ:Sontag} to this setting~\cite[Thm.~2.7]{ref:rifford2002semiconcave} and \textit{get} again an explicit feedback.

The existence of a smooth CLF is not only topologically strong, it implies there exists a \textit{robustly} stabilizing feedback~\cite{ref:ledyaev1999lyapunov}.

\subsection{On learning-based stabilization}
\label{sec:stab:project}
Neural networks are becoming increasingly popular in the context of controller synthesis~\cite{ref:jin2020neural,ref:gaby2021lyapunov,ref:mukherjee2022neural,ref:zhang2022neural}. A principled approach, however, that guarantees some form of stability is largely lacking. Progress has been made when it comes to handling side-information~\cite{ref:ahmadi2020learn}, obtaining statistical stability guarantees \cite{ref:boffi2021learn}, in the context of input-state stability~\cite{ref:yang2022input}, in the context of input-output stability by exploiting the Hamilton-Jacobi inequality~\cite{ref:okamoto2022learning}, by exploiting contraction theory~\cite{ref:rezazadeh2022learning} and by exploiting Koopman operator theory~\cite{ref:zinage2022neural}, to name a few.  
As these methods are data-driven, errors inevitably slip in and great care must be taken when one aims to mimic CLF-based controllers,~\textit{i.e.}, if $L_gV(x)=0\implies L_fV(x)<0$ holds for the estimated system, does it hold for the real system and what happens if it does not? In particular, recall~\eqref{equ:Sontag}. Moreover, in this setting the underlying dynamical control system is frequently unknown and a function class for $V$ needs to be chosen \textit{a priori}, what does this choice imply? These questions inspired this work. 

We also point out that these methods continue a long history of research on computational methods for Lyapunov functions,~\textit{e.g.}, see~\cite{ref:giesl2015review} for a review.  

\section{Topological perspective on level sets and singularities}
\label{sec:topo}
We start by detailing (recalling) how level sets of smooth Lyapunov functions, with respect to points, look like topologically. This result has some ramifications and provides for motivation in the next section.  
For simplicity, we momentarily focus on~\eqref{equ:basic:control:sys}.

In Section~\ref{sec:stab:project} we discussed why one might be interested in studying terms of the form $L_gV(x)^{-1}=\langle \nabla V(x), g(x) \rangle^{-1}$~\textit{cf.}~\eqref{equ:Sontag}. In this section we show that for practical purposes, the properties of $V$ frequently obstruct this term to be well-behaved. Indeed, singularities are studied and shown to be unavoidable when $g(x):=g$ for some $g\in \mathbb{R}^n$. 

To start, consider a $C^{0}$ dynamical system of the form~\eqref{equ:dyn:sys} on $\mathbb{R}^n$, with $n\geq 2$, and assume that $0\in \mathbb{R}^n$ is globally asymptotically stable (and hence isolated). This implies that there is a (strict) $C^{\infty}$ Lyapunov function $V:\mathbb{R}^n\to \mathbb{R}_{\geq 0}$. In particular, this implies that $V$ is also a Lyapunov function for the $C^{\infty}$ auxiliary system
\begin{equation}
\label{equ:aux}
    \frac{\mathrm{d}}{\mathrm{d}t}z(t) = - \nabla V(z(t)).
\end{equation}
Hence, $0\in \mathbb{R}^n$ is also GAS under~\eqref{equ:aux}. By a classical topological result largely\footnote{Earlier comments can be found in~\cite{ref:bobylev1974deformation}, see also~\cite{ref:jongeneel2023topological}.} due to Krasnosel'ski\u{\i} \& Zabre\u{\i}ko~\cite[Sec.~52]{ref:krasnosel1984geometrical} this directly implies that the corresponding \textit{vector field index} (with respect to $0$) satisfies 
\begin{equation*}
    \mathrm{ind}_0(-\nabla  V)=(-1)^n\neq 0. 
\end{equation*}
As the vector field index is the (oriented) degree of the map $v:\partial U \to \mathbb{S}^{n-1}$ for any open neighbourhood $U$ of $0$ containing no other equilibrium points in its closure~\cite[Sec.~6]{ref:milnor65},~\cite[Ch.~3]{ref:guillemin2010differential}, this can only be true if 
\begin{equation*}
    v:\partial U\ni z \mapsto \frac{-\nabla  V(z)}{\| \nabla  V(z)\|_2}
\end{equation*}
is surjective. As $U$ is arbitrary, it follows that the (normalized) gradient of $V$ along any non-trivial level set hits any vector in $\mathbb{S}^{n-1}$. Differently put, fix any $g\in \mathbb{S}^{n-1}$ then, for any $c>0$ there is always a $z\in V^{-1}(c)=: V_c\subset  \mathbb{R}^n$ such that $\langle \nabla  V(z), g \rangle =0$. Indeed, this is why we assumed $n\geq 2$, otherwise the claim is not true~\textit{cf.}~\cite[p.~121]{ref:sontag1989universal}. Summarizing, we have shown the following---which is attributed to Wilson~\cite{ref:wilson1967structure} and Byrnes~\cite[Thm.~4.1]{ref:byrnes2008brockett}.  

\begin{proposition}[Level sets of smooth Lyapunov functions (Wilson, Byrnes)]
\label{prop:Lyapunov:levelset}
Let $n\geq 2$ and fix some $g\in \mathbb{R}^n\setminus \{0\}$. Then, for any level set $V_{c}$, with $c>0$, of any $C^{\infty}$-smooth Lyapunov function $V:\mathbb{R}^n\to \mathbb{R}_{\geq 0}$, asserting $0\in \mathbb{R}^n$ to be GAS under some dynamical system~\eqref{equ:dyn:sys}, there is an $x\in V_c$ such that $\langle \nabla V(x),g \rangle =0$.  
\end{proposition}

Note, Proposition~\ref{prop:Lyapunov:levelset} implicitly assumes that $0\in \mathbb{R}^n$ is the only equilibrium point as we assume the origin is \textit{globally} asymptotically stable. If desired, one can adapt the statement and work with the domain of attraction. Also note that the discussion above detailed that the normalized vector $\nabla V(x)$ will hit \textit{any} vector in $\mathbb{S}^{n-1}$, our focus on $\langle \nabla V(x),g \rangle $ being equal to $0$ at some point is purely application-driven. Moreover, we see that the set of points that render the inner product zero is of codimension $1$. 

Indeed, Proposition~\ref{prop:Lyapunov:levelset} is itself classical as this result can also be understood more intuitively by directly appealing to work by Wilson. Namely, due to the work by Wilson, and later Perelman, we know that the level sets of (strict and proper) $C^{\infty}$ Lyapunov functions $V:\mathbb{R}^n\to \mathbb{R}_{\geq 0}$ are homeomorphic to $\mathbb{S}^{n-1}$~\cite{ref:wilson1967structure,ref:stillwell2012poincare}. Although we might assume that these level sets $V_c$ and $\mathbb{S}^{n-1}$ come equipped with a smooth structure, this does not immediately imply the manifolds are diffeomorphic,~\textit{e.g.}, consider Milnor's \textit{exotic spheres}~\cite{ref:Milnor7}.  Nevertheless, one expects that the gradient of $V$ along $V_c$ hits any direction when seen as a vector in $\mathbb{S}^{n-1}$, as indeed succinctly shown above. Visualizations can be found in~\cite{ref:sontag1999stability} and further comments of this nature are collected by Byrnes in~\cite{ref:byrnes2008brockett}, in particular, the diffeomorphism question is addressed.   

The ramifications for smooth CLFs are immediate as one observes that the argument with respect to the auxiliary system~\eqref{equ:aux} extends \textit{mutatis mutandis}. 

This work is motivated by renewed interest in CLFs from the neural network community. The following example highlights some work that arguably would benefit from Proposition~\ref{prop:Lyapunov:levelset}.
\begin{example}[(Almost) Singular CLF-based controllers]
In~\cite[Sec.~IV]{ref:kashima2022learning} the authors consider a dynamical control system of the form $\dot{x}=f(x)+gu$ with $f\in C^{\infty}(\mathbb{R}^2;\mathbb{R}^2)$, $g\in \mathbb{R}^2$ and $u\in \mathbb{R}$. Their to-be-learned CLF is of the form $V(x)=\sigma_k(\gamma(x)-\gamma(0))+\varepsilon \|x\|^2$ with $\gamma:\mathbb{R}^n\to \mathbb{R}$ being an input-convex neural network and $\sigma_i:\mathbb{R}\to \mathbb{R}_{\geq 0}$ { $C^{1}$-smooth locally quadratic activation functions~\cite[Eq.~(13)]{ref:kolter2019learning}} for $i=0,\dots,k$. Hence, $V\in C^{1}(\mathbb{R}^n;\mathbb{R}_{\geq 0})$. Indeed, the authors report that the learned CLF leads to large control values (under a Sontag-type controller~\eqref{equ:Sontag}), they do not detail why. The above discussion provides a topological viewpoint.  
\end{example}

One can also interpret Proposition~\ref{prop:Lyapunov:levelset} through the lens of feedback linearization.
Consider some input-output system $\Sigma$ of the form
\begin{equation}
\label{equ:IO}
\Sigma: \left\{     \begin{aligned}
\frac{\mathrm{d}}{\mathrm{d}t}{x(t)} &= f(x(t))+g u\\
y(t) &= h(x(t))
    \end{aligned}\right.
\end{equation}
for $h=V$, that is, $h$ is given by the CLF $V$ (with respect to $f$ and $g$). Let the desired output be $y_d\equiv 0$ such that $e(t)=y(t)-y_d(t)=y(t)$. Hence, $\dot{e}=\dot{V}$. Now the standard (relative degree $1$) feedback linearizing controller for~\eqref{equ:IO} is of the form $u=(L_gV)^{-1}(v-L_fV)$ with $v$ denoting the new auxiliary input~\cite{ref:isidori1985nonlinear,nonlin}. Indeed, under the choice
\begin{equation*}
    v=-\sqrt{(L_fV)^2+(L_gV)^4}
\end{equation*}
one recovers Sontag's controller~\eqref{equ:Sontag}. Now Proposition~\ref{prop:Lyapunov:levelset} tells us that the \textit{decoupling} term $(L_gV)^{-1}$ must be singular in any sufficiently small neighbourhood of $0$,~\textit{i.e.}, the relative degree assumption fails to hold.  

\begin{remark}[Generalizations]
\label{rem:gen}
To go beyond input vector fields of the form $g(x)\equiv g\in \mathbb{R}^n$ we look at two scenarios. \\
(g-i)\,\,(Dependency on $x$): Introduce the function class 
\begin{align*}
    \mathscr{G}_n:=\{g\in C^{0}(\mathbb{R}^n;\mathbb{R}^n):g(x)=g_1+g_2(x)&,\\
    g_1\in \mathbb{R}^n\setminus\{0\},\,\lim_{x\to 0}g_2(x)=0&\}. 
\end{align*}
Indeed, for any $g\in \mathscr{G}_{n}$, with $n>1$, it follows that for sufficiently small $c>0$ there is an $x\in V_c$ such that $\langle \nabla V(x),g(x)\rangle=0$. The reason being that since $g\in \mathscr{G}_{n}$ there are always $x_1,x_2\in V_c$ such that $\langle \nabla V(x_1),g(x_1)\rangle<0$ while $\langle \nabla V(x_2),g(x_2)\rangle>0$. Then the claim follows from standing regularity assumptions and the intermediate value theorem. \\
    (g-ii)\,\,(Multidimensional input): Assume that $u\in \mathbb{R}^m$ with $1<m<n$ and let the dynamical control system be of the form $\dot{x}=f(x)+\sum^m_{i=1}g_i u_i$ (dependence on $x$ can be generalized as in (g-i)). Then, as $\mathrm{span}\{g_1,\dots,g_m\}\neq \mathbb{R}^n$ there is a nonzero $v\in \mathrm{span}\{g_1,\dots,g_m\}^{\perp}$.    
\end{remark}

Exploiting the remark from above, we recover a slightly weaker version of a well-known result~\textit{cf.}~\cite[Prop.~6.1.4]{ref:bloch}, better yet, one recovers (locally) a weaker version of the highly influential obstruction to continuous asymptotic stabilization of Brockett's nonholonomic integrator~\textit{e.g.}, see~\cite[Ex.~5.9.16]{ref:sontag2013mathematical}.

{
\begin{corollary}[Obstruction for nonholonomic systems]
Assume that $u\in \mathbb{R}^m$ with $1<m<n$ and let the dynamical control system be of the form $\dot{x}=\sum^m_{i=1}g_i(x) u_i$ with $g_i\in \mathscr{G}_n$ for $i=1,\dots,m$, then, there is no smooth CLF with respect to $0\in \mathbb{R}^n$.
\end{corollary}
\begin{proof}
Indeed, this result follows from, for example, Brockett's condition~\cite{ref:brockett1983asymptotic}. However, from Proposition~\ref{prop:Lyapunov:levelset} and Remark~\ref{rem:gen} we know there is a point $x'\in \mathbb{R}^n\setminus \{0\}$ such that $\langle \nabla V(x'),\sum^m_{i=1}g_i(x')\rangle =0$. This implies that $L_fV(x')<0$ must hold for $V$ to be a CLF. As $f \equiv 0$, this is impossible and no smooth CLF can exist. 
\end{proof}
}
\section{On convexity}
\label{sec:cvx}
The previous section illustrated why level sets of Lyapunov functions are topological spheres. As such, this motivates the hope that all those Lyapunov functions can be transformed---in some sense---to the canonical Lyapunov function $V(x)=\tfrac{1}{2}\langle x,x\rangle$. {Indeed, Gr\"une, Sontag \& Wirth~\cite{grune1999asymptotic} showed that when $V$ is a $C^{\infty}$ Lyapunov function corresponding to $0$ being GAS, then, there is a $C^1$ homeomorphism $T$ such that $\widetilde{V}(T(y))=V(y)$ for $\widetilde{V}(y)=\tfrac{1}{2}\langle y,y \rangle$. However, it is not clear if their arguments can be extended to construct a homotopy from $-\nabla V(y)$ to $-y$ along vector fields such that $0$ remains GAS throughout. The complication here is the topology of the homeomorphism- and diffeomorphism groups used in their line of arguments. Those spaces are not necessarily path-connected, similar to $\{X\in \mathbb{R}^{n\times n}:\mathrm{det}(X)\neq 0\}$ not being path-connected, see~\cite[Ch.~9]{kupers2019lectures}.} 

To continue, we start this study of transformations by looking at \textit{convex} Lyapunov functions, as this class is particularly simple to handle. Better yet, by exploiting this structure, it follows that any convex Lyapunov function also asserts stability of the ``\textit{canonical}'' inward pointing vector field on $\mathbb{R}^n$ indeed, which we will denote with some abuse of notation by the map $-\mathrm{id}_{\mathbb{R}^n}$,~\textit{i.e.}, giving rise to $\dot{x}=-x$. Exactly this observation will be formalized and further studied below. 

\subsection{Convex Lyapunov functions}
Convexity in the context of Lyapunov stability theory has been an active research area.
For example, convexity in linear optimal control~\cite{AREbook}, convexity in the dual density formulation due to Rantzer~\cite{ref:prajna2004nonlinear}, convexity of the set of Lyapunov functions due to Moulay~\cite{ref:moulay2010some} {and recently, component-wise convexity of vector fields to construct Chetaev functions due to Sassano \& Astolfi~\cite{sassano2023role}}. We are, however, interested in understanding convexity of Lyapunov functions themselves. It is known that simple asymptotically stable dynamical systems do not always admit polynomial Lyapunov functions. For instance
\begin{equation}
\label{equ:ahmadi}
    \frac{\mathrm{d}}{\mathrm{d}t}\begin{pmatrix}
        x_1(t)\\ x_2(t)
    \end{pmatrix} = \begin{pmatrix}
        -x_1(t)+x_1(t)x_2(t)\\ -x_2(t)
    \end{pmatrix}
\end{equation}
does not admit a (global) polynomial Lyapunov function~\cite{ref:ahmadi2011globally}, but one can show that $V(x)=\log(1+x_1^2)+x_2^2$ is a Lyapunov function asserting $0\in \mathbb{R}^2$ is GAS. Indeed, $V$ is smooth, yet \textit{not} convex. We will come back to this several times below.
Similar obstructions can be found for analytic or rational Lyapunov functions~\cite{ref:bacciotti2005liapunov,ref:ahmadi2018globally}. 

The (computational) assumption to look for \textit{convex} Lyapunov functions is a popular one in the learning community,~\textit{e.g.}, propelled by~\cite{ref:amos2017input,ref:kolter2019learning}. However, this assumption evidently restricts the problem class that can be handled. The ramifications of assuming Lyapunov functions to be convex are understood in the context of linear systems, even for linear differential inclusions~\cite{ref:goebel2006conjugate} and linear switched systems~\cite{ref:mason2023universal}, {but not completely in the $C^0$ nonlinear setting. An exception is~\cite{ref:ahmadi2018sos}, where the authors consider nonlinear difference inclusions of the form $x_{k+1}\in \mathrm{conv}\{f_1(x_k),\dots,f_n(x_k)\}$ with $k\in \mathbb{N}$, $f_i\in C^0(\mathbb{R}^n;\mathbb{R}^n)$ $f_i(0)=0$ for $i=1,\dots, n$ and $\mathrm{conv}(\cdot)$ denoting the convex hull. Then, assuming that the maps $f_1,\dots,f_n$ share a common \textit{convex} Lyapunov function allows for concluding on $0$ being GAS\footnote{This should be understood in the discrete-time sense.}. Concurrently, they show that relaxing convexity is not possible in general, that is, counterexamples exist~\cite[Ex.~1]{ref:ahmadi2018sos}.}

Similarly, in our setting, for $n>1$, one can construct vector field examples $\dot{x}=F(x)$ over $\mathbb{R}^n$ such that $0$ is globally asymptotically stable, $F$ is smooth, yet no smooth convex Lyapunov function exists. To see why, for the sake of contradiction, one can exploit that by convexity we must have 
$
    \langle \nabla V(x), x\rangle \geq 0$ $\forall x\in \mathbb{R}^n
$
and due to the stability assumption we have
$
    \langle \nabla V(x),F(x)\rangle <0$ $ \forall x\in \mathbb{R}^n\setminus\{0\}
$
such that the function $V$ must satisfy $\langle \nabla V(x), F(x)-x \rangle < 0$ for all $x\in \mathbb{R}^n\setminus \{0\}$. Hence, if there is a non-zero fixed point\footnote{Note, here we heavily exploit the underlying vector space structure to be able to compare $x$ and $F(x)$.} of $F$, we contradict the existence of such a $V$. {See Figure~\ref{fig:ncvxLyap} for a phase portrait illustrating a dynamical system with a fixed point obstructing the existence of a \textit{smooth, convex} Lyapunov function.} As one will be able to infer from the results below, $F$ cannot point (radially) outward. Indeed, it is known that for \textit{homogeneous} Lyapunov functions this can also not be true~\cite[Prop.~1]{ref:sepulchre1996homogeneous}. {We also remark that for convex Lyapunov functions Property~(V-\ref{prop:iii:V}) is implied by Property~(V-\ref{prop:i:V})~\cite[Lem.~4.1]{ref:ahmadi2018sos}.}

\begin{figure}[t]
     \centering
         \includegraphics[scale=0.4]{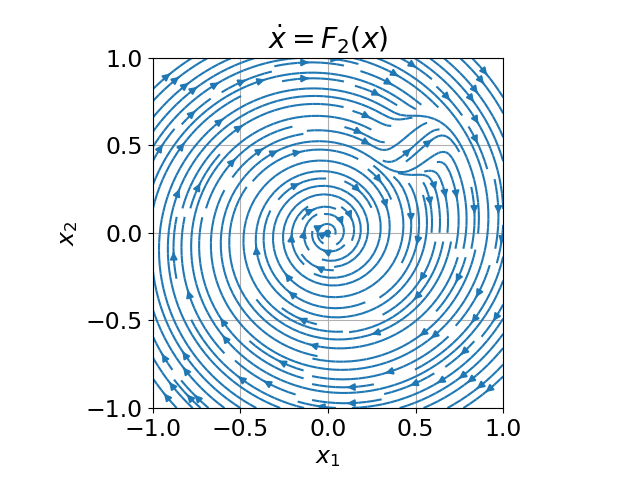}
        \caption{Example~\ref{ex:nec:ncvx}: integral curves of a smooth dynamical system that obstructs the existence of a smooth \textit{convex} Lyapunov function, yet, $0$ is GAS. Figure made with Python.}
        \label{fig:ncvxLyap}
\end{figure}

We will now formalize this observation. {To do so, we introduce the notion of a homotopy. The functions $f,g\in C^0(X;Y)$ are said to be \textit{homotopic} when there is a continuous map $H:[0,1]\times X\to Y$ such that for any $x\in X$ we have that $x \mapsto H(0,x)=f(x)$ while $x\mapsto H(1,x)=g(x)$. The homotopy is said to be a \textit{straight-line homotopy} when $H$ is simply of the form $H(s,x)=(1-s)f(x)+sg(x)$. Note that homotopies only become interesting beyond $X=Y=\mathbb{R}^n$,~\textit{e.g.}, on manifolds \textit{or} when requiring more structure to be preserved along the homotopy, as is done in this article. See~\cite{ref:jongeneel2023topological} for more on homotopies in the context of control theory. See also Section~\ref{sec:homotopy:path} for more on how homotopies allow us to discuss path-connectedness.} 

\begin{theorem}[Convex Lyapunov functions]
\label{thm:cvx:Lyap}
    Let $F\in C^{0}(\mathbb{R}^n;\mathbb{R}^n)$ give rise to $\dot{x}=F(x)$ with $0\in \mathbb{R}^n$ globally asymptotically stable (GAS) under $F$. Then, if there is a convex $C^{\infty}$ Lyapunov function asserting $0$ is GAS, the vector field $F$ is straight-line homotopic to $-\mathrm{id}_{\mathbb{R}^n}$ such that $0$ is GAS throughout the homotopy.
\end{theorem}
\begin{proof}
    {By assumption there is a $C^{\infty}$ Lyapunov function $V$ such that $\langle \nabla V(x),F(x) \rangle <0$ for all $x\in \mathbb{R}^n\setminus\{0\}$. By the convexity of $V$ we also know that
    \begin{equation}
        \label{equ:V:cvx}
        V(y)\geq V(x) + \langle \nabla V(x),y-x \rangle,\quad \forall y,x\in \mathbb{R}^n.
    \end{equation}
    In particular,~\eqref{equ:V:cvx} must hold for $y=0$, which yields $\langle \nabla V(x),-x\rangle \leq- V(x)$, that is, $V$ is also a Lyapunov function for $\dot{x}=-x$. 
    Hence, we find that $0$ is also GAS under $sF(x)-(1-s)x$ for all $s\in [0,1]$ since for any such $s$
    \begin{equation*}
        \langle \nabla V(x), sF(x)-(1-s)x\rangle <0 \quad \forall x\in \mathbb{R}^n\setminus \{0\}.
    \end{equation*}
Hence, $H(s,x)=sF(x)-(1-s)x$ is the homotopy.} 
\end{proof}
To illustrate the homotopy resulting from Theorem~\ref{thm:cvx:Lyap}, two vector fields $F_1$ and $F_2$ on $\mathbb{R}^n$ such that $0$ is GAS---asserted via possibly different smooth convex Lyapunov functions---are homotopic through a continuous map $H:[0,1]\times \mathbb{R}^n\to \mathbb{R}^n$ of the form
\begin{equation*}
    H(s,x) = \begin{cases}
-2sx +(1-2s)F_1(x) \quad & s\in [0,\tfrac{1}{2}]\\
-(2-2s)x+(2s-1)F_2(x) \quad & s\in (\tfrac{1}{2},1].
    \end{cases}
\end{equation*}
A variety of known topological conditions capture the existence of a (local) homotopy (in far more general settings), but not that \textit{along} the homotopy stability is preserved~\textit{cf.}~\cite{ref:kvalheim2022obstructions}. 

Similar statements can be made about \textit{control} Lyapunov functions.
\begin{corollary}[Convex control Lyapunov functions]
\label{cor:cvx:control:Lyap}
    Let $f\in C^{0}(\mathbb{R}^n\times \mathbb{R}^m;\mathbb{R}^n)$ give rise to the control system $\dot{x}=f(x,u)$. If there is a convex control Lyapunov function (CLF) $V\in C^{\infty}(\mathbb{R}^n;\mathbb{R}_{\geq 0})$ for this control system with respect to $0$, then, $V$ is a CLF for any control system on the straight-line homotopy between $f$ and the map $(x,u)\mapsto -x$. 
\end{corollary}
\begin{proof}
    The proof is identical to that of Theorem~\ref{thm:cvx:Lyap}, yet, now we start from $V$ satisfying
    \begin{equation*}
    \forall x\in \mathbb{R}^n\setminus\{0\}\,\exists u\in \mathbb{R}^m\,:\,\langle \nabla V(x), f(x,u)\rangle < 0.
    \end{equation*}
    and again exploit convexity of $V$ to conclude. 
\end{proof}

As remarked above, we see from Theorem~\ref{thm:cvx:Lyap} that a necessary condition for $\dot{x}=F(x)$ to admit a smooth, convex Lyapunov function, asserting $0$ is GAS, is that
\begin{equation}
    \label{equ:cvx:L:nec}
   F(x)\neq \lambda x \quad \forall \lambda\in \mathbb{R}_{\geq 0},\, \forall x\in \mathbb{R}^n\setminus\{0\}. 
\end{equation}
Differently put, if $V$ is a Lyapunov function for $\dot{x}=F(x)$, it is also a Lyapunov function for $\dot{x}=F(x)-\lambda x$, with $\lambda\geq 0$. The next example shows we can find families of dynamical systems that do not obey Condition~\eqref{equ:cvx:L:nec}. 

\begin{example}[Necessarily nonconvex]
\label{ex:nec:ncvx}
The system as shown in Figure~\ref{fig:ncvxLyap} can be made explicit. Consider a $C^{\infty}$ dynamical system of the form~\eqref{equ:dyn:sys} on $\mathbb{R}^2$ as given by
\begin{equation}
\label{equ:nlin:ncvx}
\begin{aligned}
    \frac{\mathrm{d}}{\mathrm{d}t}\begin{pmatrix}
    x_1(t)\\ x_2(t)
\end{pmatrix} = \begin{pmatrix}
 \alpha & 1\\
    -1 & \alpha
\end{pmatrix}\begin{pmatrix} 
    x_1(t)\\ x_2(t)
\end{pmatrix}& \\
+  \gamma \left(\mathrm{exp}(-\beta \|x(t)-p\|_2^2 )-\mathrm{exp}(-\beta \|p\|_2^2 )\right)&\begin{pmatrix}
   1\\ 1
\end{pmatrix}
\end{aligned}
\end{equation}
where $x=(x_1,x_2)\in \mathbb{R}^2$, and $\alpha=-0.1$, $\beta=100$, $\gamma=10$ and $p=(0.5,0.5)\in \mathbb{R}^2$ correspond to Figure~\ref{fig:ncvxLyap}. Indeed, it can be shown that only $0\in \mathbb{R}^2$ is an equilibrium point of this dynamical system. Note,~\eqref{equ:nlin:ncvx} can also be understood as a stabilizable linear system under a (bounded) nonlinear perturbation. Regarding our necessary condition for convexity, we find, for instance, that $x=(0.51,0.45)\in \mathbb{R}^2$ and $\lambda=1/0.0629$ provide for a (numerical)\footnote{{We write ``\textit{numerical}'' since the obstruction is not provided in closed-form. However, our argument is constructive and as follows. Due to the symmetry in the nonlinear part, a necessary condition for~\eqref{equ:nlin:ncvx} to satisfy $F_2(x)=\lambda x$ is that $(\lambda-\alpha-1)x_1=(\lambda-\alpha+1)x_2$ holds for some appropriate tuple $(\lambda,\alpha,x_1,x_2)$. For instance, fix $(\alpha,\lambda,x_2)$ such that $(\lambda-\alpha-1)\neq 0$ and solve for $x_1$. Denote the resulting $(x_1,x_2)$ by $\bar{x}$. Next, let $\zeta=\mathrm{exp}(-\beta \|\bar{x}-p\|_2^2 )-\mathrm{exp}(-\beta \|p\|_2^2 )$, then, we can simply set $\gamma=(\lambda x_1-\alpha x_1 - x_2)/\zeta$ such that $F_2(\bar{x})=\lambda \bar{x}$ holds. The provided numerical values are obtained accordingly and hence provide for a valid obstruction.}} invalidation of~\eqref{equ:cvx:L:nec} indeed. {Regarding stability, we remark that $\gamma$ can be understood as a bifurcation parameter, with $0$ being GAS for our choice of $\gamma$}\footnote{{To be more precise, $0$ is GAS for $\gamma \in [0,11.2782)$, while for $\gamma = 11.2782$, $0$ is just locally asymptotically stable, enclosed by a semi-stable periodic orbit. Then for $\gamma>11.2782$, the semi-stable orbit bifurcates into an unstable- and a stable orbit. The analysis is performed using MatCont 7.4~\cite{ref:dhooge2008new}. A simulation can be found at \url{https://wjongeneel.nl/bifurcationF2.gif}. This type of bifurcation appears in neuroscience~\cite[Fig. 1.7]{ref:izhikevich2001synchronization} and goes by several names, \textit{e.g.}, a saddle-node bifurcation of a limit cycle or a fold bifurcation of a cycle, see also~\cite[Sec. 8.4]{ref:strogatz}.}}.
\end{example}

\begin{remark}[A nonconvex conic structure]
    Let $0\in \mathbb{R}^n$ be GAS under $\dot{x}=F(x)$, then $0$ is also GAS under $\dot{x}=\theta F(x)$ for any $\theta>0$,~\textit{e.g.}, consider $\langle \nabla V(x), F(x)\rangle $ and $\langle \nabla V(x), \theta F(x)\rangle$ for some Lyapunov function $V$ with respect to $F$. Hence, {if $F$ is convex}, then by Theorem~\ref{thm:cvx:Lyap}, all $\theta F$ are straight-line homotopic to $-\mathrm{id}_{\mathbb{R}^n}$. {Despite the conic structure, convexity of the set of these vector fields (vector fields such that $0$ is GAS)} breaks down as already the set of Hurwitz stable matrices is nonconvex,~\textit{e.g.}, 
    \begin{equation*}
        s\begin{pmatrix}
            -1 & 10\\ 0 & -1
        \end{pmatrix} + (1-s)\begin{pmatrix}
            -1 & 0\\ 10 & -1
        \end{pmatrix}
    \end{equation*}
    becomes unstable (not all eigenvalues lie in $\mathbb{C}_{\Re<0}$) for $s=\tfrac{1}{2}$. 
\end{remark}

Similarly, from Corollary~\ref{cor:cvx:control:Lyap} we see that the control system $\dot{x}=f(x,u)$ admits a smooth, convex CLF only when
\begin{equation}
    \label{equ:cvx:CLF:nec}
    \forall x\in \mathbb{R}^n\setminus \{0\}\,\exists u \in \mathbb{R}^m\,:\,  f(x,u)\neq \lambda x \quad \forall \lambda\in \mathbb{R}_{\geq 0}.  
\end{equation}
Indeed, one can replace $\lambda\in \mathbb{R}_{\geq 0}$ in~\eqref{equ:cvx:CLF:nec} by, for example, $\lambda\in C^{0}(\mathbb{R}^n;\mathbb{R}_{\geq 0})$~\textit{cf.}~\cite{ref:sepulchre1996homogeneous}. 

\begin{example}[Linear dynamical systems]
\label{ex:lin:sys}
    Consider the linear dynamical system $\dot{x}=Ax$ for some matrix $A\in \mathbb{R}^{n\times n}$. Theorem~\ref{thm:cvx:Lyap} implies that for a convex Lyapunov function to exist (asserting $0$ is GAS) the expression $sAx-(1-s)x$ cannot vanish for some $s\in[0,1]$ and $x\in \mathbb{R}^n\setminus 0$. Reformulating, we get~\eqref{equ:cvx:L:nec},~\textit{i.e.}, $Ax=\lambda x$ cannot have a solution for some $\lambda\geq 0$ and $x\in \mathbb{R}^n\setminus \{0\}$. However, this is precisely stating that $A$ cannot have an unstable eigenvalue of the form $\lambda\in \mathbb{R}_{\geq 0}$. Indeed, for globally asymptotically stable linear systems a convex (quadratic) Lyapunov function of the form $V(x)=\tfrac{1}{2}\langle Px, x \rangle$ always exists~\cite[Thm.~18]{ref:sontag2013mathematical}.  
\end{example}

\begin{remark}[On sufficiency]
\label{rem:suff}
 Example~\ref{ex:nec:ncvx} showed that for dynamical systems of the form~\eqref{equ:nlin:ncvx} no convex Lyapunov function can exist. Going back to~\eqref{equ:ahmadi}, the provided Lyapunov function is nonconvex. Concurrently, one can check that~\eqref{equ:cvx:L:nec} holds, so a convex Lyapunov function is not ruled out. We come back to this below.   
\end{remark}

To elaborate on Example~\ref{ex:lin:sys}, for controllable linear systems,~\textit{e.g.}, of the form $\dot{x}=Ax+Bu$, one can always parameterize a quadratic Lyapunov function for the LQ optimally controlled closed-loop system by the positive definite solution to the corresponding Riccati equation (for any appropriate cost)~\cite[Thm.~42, Ex.~8.5.4]{ref:sontag2013mathematical}.

\begin{example}[Linear dynamical control systems and Hautus' test]
\label{ex:lin:control:sys}
A celebrated condition largely attributed to Hautus (plus Belevitch and Popov) states that a linear dynamical control system of the form $\dot{x}=Ax+Bu$, is stabilizable when 
\begin{equation}
    \label{equ:Hautus}
    \mathrm{rank}\left( \begin{pmatrix}A-\lambda I_n & B \end{pmatrix} \right)=n \quad \forall \lambda \in \sigma(A)\cap \mathbb{C}_{\Re\geq 0}, 
\end{equation}
where $\sigma(A)$ denotes the spectrum of $A$. See for instance~\cite[Ch.~3]{ref:trentelman2012control}.
Now, elementary algebraic arguments show that Hautus' condition~\eqref{equ:Hautus} implies that~\eqref{equ:cvx:CLF:nec} holds, as it should for linear control systems. 
\end{example}

Using the above, one can readily verify that, for example 
\begin{equation}
\label{ex:sys:1}
\frac{\mathrm{d}}{\mathrm{d}t}
    \begin{pmatrix}
        {x_1}(t)\\ {x_2}(t)
    \end{pmatrix} = \begin{pmatrix}
        x_1(t) u\\ x_1(t)x_2(t) u
    \end{pmatrix}
\end{equation}
does not admit a smooth, convex CLF. Indeed, for~\eqref{ex:sys:1}, controllability is lost at $(0,x_2)\in \mathbb{R}^2$. 

Example~\ref{ex:lin:sys} and Example~\ref{ex:lin:control:sys} show that conditions~\eqref{equ:cvx:L:nec} and~\eqref{equ:cvx:CLF:nec} are to some extent generalizations of known conditions for linear systems, yet, lifted to nonlinear systems under convexity assumptions. These conditions are, however, weak. 

A stronger set of conditions one can derive from Theorem~\ref{thm:cvx:Lyap} is of the form: $\dot{x}=f(x,u)$ admits a smooth, convex CLF only if $\dot{x}=f(x,u)-\lambda(x)x$ does, for any $\lambda\in C^{0}(\mathbb{R}^n;\mathbb{R}_{\geq 0})$.
We are not the first to observe something of this form, \textit{e.g.},~Sepulchre \& Aeyels~ \cite[Sec.~4.1]{ref:sepulchre1996homogeneous} look at homogeneous CLFs and recover a similar condition. 

{
 We close this subsection with a comment on mere \textit{Lyapunov stability}, {that is, Property~(V-\ref{prop:ii:V}) is replaced with the weaker notion $\langle \nabla V(x),F(x) \rangle \leq 0$. Although this notion of stability is understood as local, under sub-level set compactness, such a Lyapunov function is of use as it allows for concluding trajectories to remain bounded.}  
\begin{remark}[Lyapunov stability]
 At the time of writing, several examples of Lyapunov stable dynamical systems surfaced that provably fail to admit a smooth, convex Lyapunov function~\cite{ref:akbarian2023introducing}. We show that our line of arguments offers an arguably simpler means of reaching such a conclusion. Following the same reasoning as for Theorem~\ref{thm:cvx:Lyap}, when $0$ is Lyapunov stable under some vector field $F$ and comes equipped with a $C^{\infty}$ convex Lyapunov function $V$, then, we must have that $\langle \nabla V(x), F(x) - \lambda x\rangle \leq 0$ $\forall x\in \mathbb{R}^n,\, \lambda \in \mathbb{R}_{\geq 0}$. We claim that the existence of a point $x'\in \mathbb{R}^n\setminus\{0\}$ such that $F(x')_i = \lambda' x_i' \neq 0$ for $i=1,\dots,n$ for some $\lambda'\in \mathbb{R}_{>0}$ contradicts the existence of such a function $V$. To see this, suppose that such a pair $(x',\lambda')$ exists, then we can find $\lambda_1,\lambda_2 \in \mathbb{R}_{>0}$ such that $2F(x')=(\lambda_1+\lambda_2)x'$. In particular, we have that $F(x')-\lambda_1 x' = (-1)(F(x')-\lambda_2 x')$. We can select $\lambda_1\neq \lambda_2$ such that by construction no element of the equation above equals $0$. However, that means that when we move $\lambda$ from $\lambda_1$ to $\lambda_2$, the sign of $\langle \nabla V(x'),F(x')-\lambda x'\rangle$ flips, which is a contradiction. One can employ precisely this argument to show that for $k>1$ the origin of the system
\begin{align*}
    \frac{\mathrm{d}}{\mathrm{d}t}\begin{pmatrix}
        r(t)\\ \theta(t)
    \end{pmatrix} = \begin{pmatrix}
        r^{2k+1}\sin(1/r)\\
        1
    \end{pmatrix}
\end{align*}
is Lyapunov stable, yet, no smooth, convex Lyapunov function exists to assert this~\textit{cf.}~\cite[Thm.~1]{ref:akbarian2023introducing}.
\end{remark}
}

\subsection{On compact convex sets}
We briefly show that without too much effort the results extend from $0\in \mathbb{R}^n$ being GAS under some dynamical system parametrized by $F\in C^{0}(\mathbb{R}^n;\mathbb{R}^n)$ to a compact convex set $A\subseteq \mathbb{R}^n$ being GAS\footnote{For more on the generalization of stability notions from points to sets we point the reader to~\cite{ref:hurley1982attractors} for a topological treatment.} under $F$. As $A$ is homotopy equivalent to a point, this is perhaps not surprising. Define the projection operator by
\begin{equation*}
    \Pi_A(x):=\arg\min_{y\in A}\|x-y\|_2. 
\end{equation*}
We have the following.
\begin{corollary}[Convex Lyapunov functions for convex compact sets]
\label{cor:cvx:Lyap:A}
    Let $F\in C^{0}(\mathbb{R}^n;\mathbb{R}^n)$ give rise to $\dot{x}=F(x)$ with a compact convex set $A\subseteq \mathbb{R}^n$ being globally asymptotically stable (GAS) under $F$. Then, if there is a convex $C^{\infty}$ Lyapunov function asserting $A$ is GAS, the vector field $F$ is straight-line homotopic to $\Pi_A-\mathrm{id}_{\mathbb{R}^n}$ on $\mathbb{R}^n\setminus A$ such that $A$ is GAS throughout the homotopy.
\end{corollary}
\begin{proof}
    The Lyapunov function is such that $V(x)=0\iff x\in A$, hence for the convexity condition $V(y)\geq V(x)+\langle \nabla V(x),y-x \rangle$ we pick $y=\Pi_A(x)$ such that for all $x\in \mathbb{R}^n\setminus A$ we have $\langle \nabla V(x),\Pi_A(x)-x\rangle <0$. We can conclude. 
\end{proof}
Some comments are in place, we do not need $F(A)=0$, $A$ merely needs to be invariant\footnote{Let $\varphi$ be the flow corresponding to $F$, then $A$ is said to be invariant (under $\varphi$) when $\varphi(\mathbb{R},A)=A$.}. This is why we cannot say anything about the homotopy on $A$ itself. Moreover, settings like these easily obstruct $V\in C^{\omega}(\mathbb{R}^n;\mathbb{R}_{\geq 0})$ (real-analyticity), not to contradict real-analytic function theory (bump functions cannot be $C^{\omega}$). Also, when $A$ is not convex, $\Pi_A$ is potentially set-valued, obstructing our vector field construction and perhaps $\Pi_A-\mathrm{id}_{\mathbb{R}^n}$ is not the expected ``\textit{canonical}'' inward vector field (due to the non-scaled offset $-x$). At last, we point out that although $V$ is smooth, this does not imply that $\partial A$ must be a smooth manifold. For instance, consider $V(x_1,x_2)=(x_1-x_2)^2(x_1+x_2)^2$ (although here, $V^{-1}(0)$ is clearly not convex). We direct the reader to~\cite{ref:fathi2019smoothing} and references therein for more on Lyapunov theory with respect to sets. 

\subsection{Geodesic convexity} 
To go beyond vanilla convexity, we follow~\cite[Ch.~3]{ref:udriste2013convex},~\cite[Ch.~11]{ref:boumal2020introduction} and show how the situation is hardly different in the context of \textit{geodesic} convexity. We will be brief, for the details on geodesic convexity we point the reader to the references above and for background information on Riemannian geometry we suggest~\cite{Lee3}. 

Let $(\mathbb{R}^n,g)$ be a $C^{\infty}$ \textit{Riemannian manifold} for some Riemannian metric $g$. One can think of $g$ as inducing a change of coordinates via the inner product $\langle \cdot,\cdot \rangle_g$, in particular, this metric has an effect on gradients,~that is, the (Riemanian) gradient of a differentiable function $f:\mathbb{R}^n\to \mathbb{R}$, with respect to $g$, satisfies $Df(x)[v]=\langle \mathrm{grad}\,f(x), v\rangle_g$ for any $(x,v)\in T\mathbb{R}^n$, with $Df(x)[v]$ being the directional derivative in the direction $v\in T_x\mathbb{R}^n$. For example, let $g$ be parametrized by a symmetric positive definite matrix $P$, that is, $\langle v,w \rangle_g:=\langle Pv,w\rangle$ for any $v,w\in T_x\mathbb{R}^n$ and $x\in \mathbb{R}^n$, then, $\mathrm{grad}\,f(x) = P^{-1}\nabla f(x)$. Indeed, for a practical application of this in $\mathbb{R}^n$, we point the reader to a discussion of Newton's method as used in second-order optimization~\cite[Sec.~9.5]{ref:Boyd_04}. The metric $g$ also has ramifications for ``\textit{straight lines}'', a $C^1$ curve $[0,1]\ni s\mapsto \gamma(s)$ is a \textit{geodesic}, with respect to $g$, when it is an \textit{extremal} of the energy functional $E(\gamma):=\tfrac{1}{2}\int_{[0,1]} \langle \dot{\gamma}(\tau),\dot{\gamma}(\tau)\rangle_g \mathrm{d}\tau$. This implies geodesics are \textit{locally} minimizing length and in that sense they generalize straight lines. As this statement is local, geodesics are by no means always unique. Then, a subset $U\subseteq \mathbb{R}^n$ is called \textit{geodesically convex} ($g$-convex) when for all points $x,y\in U$ there is a \textit{unique}\footnote{See the discussion in~\cite[Sec.~11.3]{ref:boumal2020introduction} on various slightly different definitions of \textit{geodesic convexity} and their implications.} geodesic $\gamma:[0,1]\to \mathbb{R}^n$ (with respect to $g$) connecting $x$ to $y$ such that $\gamma([0,1])\subseteq U$. A function $f:U\subseteq \mathbb{R}^n\to \mathbb{R}$, over some $g$-convex domain $U$, is said to be \textbf{\textit{geodesically convex}} (\textbf{\textit{$g$-convex}}) when
\begin{equation}
    \label{equ:g:cvx:geod}
     (1-t)f(x)+tf(y)\geq f(\gamma(t))\quad \forall t\in [0,1]
\end{equation}
for $\gamma:[0,1]\to \mathbb{R}^n$ a geodesic, with $\gamma([0,1])\subseteq U$ connecting the point $x$ to $y$. Indeed,~\eqref{equ:g:cvx:geod} generalizes the standard $C^0$ definition of convexity. A $C^1$ condition is now given by
\begin{equation*}
    f(\mathrm{Exp}_x(tv))\geq f(x) + t\langle \mathrm{grad}\,f(x),v\rangle_g\quad \forall t\in [0,1],
\end{equation*}
where $v\in T_x\mathbb{R}^n$, $\mathrm{Exp}_x$ is the (Riemannian) \textit{exponential map} at $x\in U\subseteq \mathbb{R}^n$ and $\mathrm{grad}\,f(x)$ is the \textit{Riemannian gradient} of $f$. Here, the exponential map is defined, locally, by $\mathrm{Exp}_x(v)=\gamma(1)$ for $\gamma$ the unique geodesic with $\gamma(0)=x$ and $\dot{\gamma}(0)=v$.

Similarly, for a $C^2$ condition, a function $f$ is $g$-convex when the \textit{Riemannian Hessian} satisfies $\mathrm{Hess}\,f(x)\succeq 0$ for all $x\in U\subseteq  \mathbb{R}^n$. The interest in $g$-convex functions stems from the fact that local minima are again global minima, as with standard convex functions. 

We are now equipped to generalize Theorem~\ref{thm:cvx:Lyap}.

\begin{theorem}[Geodesically convex Lyapunov functions]
\label{thm:geod:cvx:Lyap}
Let $(\mathbb{R}^n,g)$ be a \textit{Riemannian manifold} and let $U\subseteq \mathbb{R}^n$ be open and $g$-convex. Let $F\in C^{0}(U;\mathbb{R}^n)$ give rise to $\dot{x}=F(x)$ with $0\in U$ globally asymptotically stable (GAS) (on $U$) under $F$. Then, if there is a g-convex $C^{\infty}$ Lyapunov function asserting $0$ is GAS, the vector field $F$ is straight-line homotopic to $\mathrm{Exp}^{-1}(0)$ such that $0$ is GAS throughout the homotopy.
\end{theorem}
In Theorem~\ref{thm:geod:cvx:Lyap}, $\mathrm{Exp}^{-1}(0)$ should be understood as the map being defined by $x\mapsto \mathrm{Exp}_x^{-1}(0)\in T_x\mathbb{R}^n$. 
\begin{proof}
By assumption, there is a $C^{\infty}$ Lyapunov function $V$ such that $\langle \nabla V(x),F(x) \rangle <0$ for all $x\in U\setminus\{0\}$. By the $g$-convexity of $V$ we also know that for all $t\in[0,1]$ and $(x,v)\in TU$ we have
\begin{align*}
    V(\mathrm{Exp}_x(tv))\geq& V(x)+t\langle \mathrm{grad}\,V(x),v\rangle_g\\
    =& V(x) + t\langle \nabla V(x),v\rangle,
\end{align*}
where we removed the dependency on the metric $g$ by identifying both inner products with the directional derivative $DV(x)[v]$. We consider $t=1$ and pick $v:=\mathrm{Exp}_x^{-1}(0)$. This map is always well-defined since our geodesics are unique. Now we proceed exactly as in the proof of Theorem~\ref{thm:cvx:Lyap} and conclude. 
\end{proof}

Indeed, we recover Theorem~\ref{thm:cvx:Lyap} for the identity metric on $\mathbb{R}^n$ and $U=\mathbb{R}^n$. In particular, in that case we can define our Riemannian exponential map as $\mathrm{Exp}_x(v)=x+v$ for (sufficiently small) $v\in T_xU$. Hence, the tangent vector $v$ such that $\mathrm{Exp}_x(v)=0$ is simply $-x$ (now seen as a tangent vector),~\textit{i.e.}, $\mathrm{Exp}_x^{-1}(0)=-x$. Recall, formally speaking, $-\mathrm{id}_{\mathbb{R}^n}$ should be understood as $x\mapsto (x,-x)\in TU$ while ignoring the first component of the image. With this in mind we can again understand $\mathrm{Exp}^{-1}(0)$ as the canonical ``\textit{inward}'' vector field, yet now on a subset of $(\mathbb{R}^n,g)$.  

Generalizing to compact manifolds and so forth (beyond contractible sets) is somewhat nonsensical as no smooth function function with a single critical point exists on those spaces. This restriction comes from the demand that our geodesics are unique, obstructing nontrivial topologies. See~\cite[Ch.~4]{ref:udriste2013convex} for more pointers. 

A similar generalization can be achieved through the lens of \textit{contraction analysis}~\cite{ref:lohmiller1998contraction}. See in particular~\cite{ref:wensing2020beyond} for a relation between $g$-convexity and contraction metrics. 

We end this section by returning to Remark~\ref{rem:suff}, the Lyapunov function with respect to~\eqref{equ:ahmadi} is nonconvex, yet the dynamical system satisfies the necessary condition~\eqref{equ:cvx:L:nec}. {Indeed, the function \textit{is} locally $g$-convex\footnote{In fact, the function $x\mapsto \log(1+x^2)$ is also semiconcave.} (under quadrant-wise exponential geodesics, using a ``\textit{log-barrier}'' metric \textit{cf.}~\cite[Ex. 4.8]{ref:vishnoi2018geodesic}, thanks to the invariance properties of the vector field). However, under such a choice of metric the exponential map is not well-defined (the metric is singular at $0$), generalizing our framework to handle weaker regularity conditions is left for future work.} 

\section{On continuation}
\label{sec:continuation}
The existence of a mere homotopy is not immediately informative. Often, only when the homotopy itself satisfies certain properties, one can draw nontrivial conclusions. 

In our case the homotopies as detailed in Theorem~\ref{thm:cvx:Lyap}, Corollary~\ref{cor:cvx:control:Lyap}, Corollary~\ref{cor:cvx:Lyap:A} and Theorem~\ref{thm:geod:cvx:Lyap} all preserve qualitative properties of the underlying dynamical system. More formally, this construction provides a \textit{continuation} in the sense of Conley, albeit from a different perspective. Again, we are decidedly brief, but we point the reader to~\cite{ref:conley1978isolated,ref:mischaikow2002conley} for more details on Conley index theory and suggest~\cite{Hatcher} as a reference on algebraic topology.

Recall that a dynamical system of the form~\eqref{equ:dyn:sys} gives rise to a global flow $\varphi:\mathbb{R}\times \mathbb{R}^n\to \mathbb{R}^n$. Let $S\subset \mathbb{R}^n$ be an \textit{isolated invariant set} (with respect to $\varphi)$, that is,
\begin{equation*}
    S = \mathrm{Inv}(M,\varphi):=\{x\in M:\varphi(\mathbb{R},x)\subseteq M\}\subseteq \mathrm{int}(M)
\end{equation*}
for some compact set $M\subset \mathbb{R}^n$. Note that not every invariant set is isolated,~\textit{e.g.} consider an equilibrium point of the \textit{center}-type. Then, a pair of compact sets $(N,L)\subset \mathbb{R}^n\times \mathbb{R}^n$ is an \textit{index pair} for $S$ when
\begin{enumerate}[({I}-i)]
    \item $S=\mathrm{Inv}(\mathrm{cl}(N\setminus L),\varphi)$ and $N\setminus L$ is a neighbourhood of $S$;
    \item $L$ is positively invariant in $N$;
    \item $L$ is an exit set for $N$ (a trajectory that leaves $N$, must leave through $L$).
\end{enumerate}
Now, the (homotopy) \textit{Conley index} of $S$ is the homotopy type of the pointed (quotient) space $(N/L,[L])$,~\textit{e.g.}, for $N=\mathbb{B}^n$, $L=\partial \mathbb{B}^n=\mathbb{S}^{n-1}$, we have that $N/L\simeq \mathbb{S}^n$ such that $(N/L,[L])$ is the pointed $n$-sphere. As this object is hard to computationally work with, let $H^k(A,B;\mathbb{Z})$ denote the $k^{\mathrm{th}}$ singular cohomology group of $A$ relative to $B\subseteq A$, then, the \textit{homological} \textbf{\textit{Conley index}} defined as $\mathrm{CH}^k(S,\varphi):=H^k(N/L,[L];\mathbb{Z})$ is of larger interest,~\textit{e.g.}, as computational tools \textit{are} available~\cite{ref:kaczynski2004computational}. Going back to our setting, assume for that moment that $0\in \mathbb{R}^n$ is a GAS hyperbolic fixed point of the flow $\varphi$. By hyperbolicity (local linearity), we can pick $N=\varepsilon\mathbb{B}^n$ (a sufficiently small closed ball in $\mathbb{R}^n$) and $L=\emptyset$. Now see that
\begin{equation*}
\begin{aligned}
\mathrm{CH}^k(0,\varphi)&=H^k(\varepsilon\mathbb{B}^n/\emptyset,[\emptyset];\mathbb{Z})\\
        &\simeq H^k(\varepsilon\mathbb{B}^n;\mathbb{Z})\\
        &\simeq \begin{cases}
        \mathbb{Z}\quad & \text{if }k=0\\
        0\quad & \text{otherwise}
    \end{cases}
\end{aligned}
\end{equation*}
since $\varepsilon\mathbb{B}^n$ is homotopic to a point. If $0$ is not hyperbolic, pick $N$ to be a sub-level set of a smooth Lyapunov function that asserts $0$ is GAS, this set is compact by Property~(V-\ref{prop:iii:V}). Indeed, constructions like these provide for topological obstructions~\cite{ref:moulay2011conley}. 

Now, if some $N$ can be chosen to be an isolating neighbourhood \textit{throughout} a homotopy, then the Conley index is preserved along that homotopy~\cite[Thm.~1.10]{ref:mrozek1994shape}. Simply put, we speak in this case of a \textit{\textbf{continuation}} between the dynamical system at the beginning and the end of the homotopy. A question asked by Conley concerns the opposite~\cite[p.~83]{ref:conley1978isolated}, to what extent do equivalent Conley indices relate to the existence of such a continuation. See also the discussion in~\cite{ref:mrozek2000conley,ref:kvalheim2022obstructions}. Indeed, we see that if there is a homotopy through flows $[0,1]\ni\lambda\mapsto \varphi_{\lambda}$ such that $0$ is GAS along the homotopy, then $\mathrm{CH}^k(0,\varphi_0)\simeq \mathrm{CH}^k(0,\varphi_1)$.

For the other direction, based on the above we have the following. One can extend the statement to compact convex sets or $g$-convexity if desired. 
\begin{corollary}[On continuation and convex Lyapunov functions]
  Let $0\in \mathbb{R}^n$ be GAS under two dynamical systems of the form~\eqref{equ:dyn:sys} parametrized by $F_0$ and $F_1$, giving rise to the flows $\varphi_0$ and $\varphi_1$. Assume that $0$ being GAS is asserted by---possibly different---smooth, convex Lyapunov functions $V_0$ and $V_1$. Then $0$ (with respect to $\varphi_0$) and $0$ (with respect to $\varphi_1$) are related by continuation where $N$ can be chosen to be of the form $N=N_0\cap N_1$ (based on sublevel sets of $V_0$ and $V_1$).  
\end{corollary}
 A further study of this observation is the topic of future work. 

To return to similarities pointed out in the introduction, the work by Reineck~\cite{ref:reineck1991continuation} and the proof of \cite[Thm.~11.4]{ref:coron2007control} provide the homotopy (preserving the Conley index) between $F$ and the (a) negative gradient flow $-\nabla V$. However, how to link---if at all--- multiple dynamical systems is unclear. The book by Cieliebak \& Eliashberg does contain results in this direction, yet under $C^k$-nearness assumptions~\cite[Ch.~9]{ref:cieliebak2012stein}, not in general. 

Then, this work alludes to convexity being a simple structural ingredient to actually link several dynamical systems together via some canonical dynamical system. 

\section{Conclusion and future work}
\label{sec:conclusion}
{
We showed that the space of dynamical systems over $\mathbb{R}^n$ with $0$ being GAS subject to the existence of a (generalized) convex Lyapunov function is path-connected, see Section~\ref{sec:homotopy:path}. As a byproduct we derived necessary conditions for smooth, convex (control) Lyapunov functions to exist.}

There is recent work regarding paths in the space of stable dynamical systems in the context of linear optimal control~\cite{ref:bu2019lqr,ref:ECCJongeneelKuhn21}, but further extensions are largely lacking. We hope this article inspires more work. 

This work focused on the complete $C^{0}$ setting with the emphasis on $\mathbb{R}^n$, future work aims at studying dynamical control systems under weaker regularity assumptions in more general spaces with the focus on more general attractors. 

Also, this work focused on the exploitation of a $g$-convex structure, however, more general structures have been proposed and studied,~\textit{e.g.}, a compositional structure~\cite{ref:Grune2021}. It seems worthwhile to study more structural assumptions along the lines of this article and previous work by Aeyels \& Sepulchre~\cite{ref:sepulchre1996homogeneous}. {For instance, one could consider \textit{weak convexity},~\textit{e.g.}, see~\cite{davis2019stochastic}, and similarly, one might consider other stability notions that provide for more structure, like \textit{exponential} stability \textit{cf.}~\cite{vidyasagar2022new}. 

Another direction of future work is to elaborate on the work by Gr\"une, Sontag \& Wirth~\cite{grune1999asymptotic}. 
In the remaining subsections we identify more concrete directions of future work. 
}

{
\subsection{Invexity}
Let $W\subseteq \mathbb{R}^n$ be open. A function $f\in C^1(W;\mathbb{R})$ is said to be {\textit{invex}} when there is a map $\eta:W\times W\to \mathbb{R}^n$ such that $f(y)\geq f(x)+\langle \nabla f(x),\eta(x,y)\rangle$ for any $x,y\in W$. Differently put, invex functions are such that critical points, \textit{i.e.}, $x^{\circ}\in W$ such that $\nabla f(x^{\circ})=0$, are \textit{global} minima. The map $\eta$ is sometimes referred to as the \textit{kernel} and $f$ is said to be invex with respect to this kernel $\eta$. Let $x^{\circ}$ be a critical point of some invex function $f\in C^1(W;\mathbb{R})$, then $f(x)\geq f(x^{\circ})$ for all $x\in X$. Conversely, let $f\in C^1(W;\mathbb{R})$ be such that every critical point is a global minimizer, then we can define $\eta:W\times W\to \mathbb{R}^n$ by
\begin{equation*}
\label{equ:eta:1}
    \eta(x,y) = \begin{cases}
0 \quad &\text{if } \nabla f(x)=0\\
\dfrac{(f(y)-f(x))\nabla f(x)}{\langle \nabla f(x),\nabla f(x) \rangle}\quad &\text{otherwise}.
    \end{cases}
\end{equation*}
This construction shows that indeed $f\in C^1(W;\mathbb{R})$ is invex if and only if every critical point is a global minimizer. So far, we have not said anything about the space of kernels, and in that sense, $\eta$ is unconstrained and not equipped with any structure. Indeed, invexity has been the subject of controversy~\cite{ref:zualinescu2014critical,ref:borwein2017generalisations}, mainly due to vacuous generalizations. Nevertheless, continuity of $\eta$ has been studied~\cite{ref:smart1996continuity} and a further study might allow for generalizing several arguments from above. A similar viewpoint can be found in~\cite{barik2023invex}, where the authors identify mild assumptions on $\eta$ such that first-order invex optimization algorithms provably converge. 
}

{
\subsection{Convex envelopes}
Suppose $0$ is GAS under a vector field $F$ on $\mathbb{R}^n$, hence, there is a $C^{\infty}$ Lyapunov function $V$ and we know that there is a homotopy between $F$ and $-\nabla V(x)$ such that along the homotopy $0$ remains GAS. Now, construct the \textit{convex envelope} of $V$ as $\mathrm{conv}({V})(x):=\sup\{g(x):g\text{ is convex and }g\leq V \text{ on }\mathbb{R}^n\}$. It can be shown that $\mathrm{conv}(V)$ is $C^1$ by our assumptions on $V$~\cite{ref:kirchheim2001differentiability}. It readily follows that, {with respect to $\dot{z}=-\nabla \mathrm{conv}(V)(z)$,} $\mathrm{conv}(V)$ satisfies Properties~(V-\ref{prop:i:V})-(V-\ref{prop:iii:V}), as such, $0$ is GAS under $\dot{z}=-\nabla \mathrm{conv}(V)(z)$. Therefore, a homotopy between $V$ and $\mathrm{conv}(V)$ that preserves regularity and invexity would allow for solving the main research question of this article for equilibrium points of vector fields on $\mathbb{R}^n$. A similar question has been studied in the context of Hamilton-Jacobi equations. Omitting details, Vese showed in 1999 that the PDE
\begin{equation}
\label{equ:vese}
    \frac{\partial u}{\partial t} = \sqrt{1+\|\nabla_x u\|_2^2 }\min \{0, \lambda_{\mathrm{min}}(\nabla_x^2 u)\}
\end{equation}
converges to precisely the convex envelope of $u(0,x)$~\cite{vese1999method}. This observation has been used in the context of homotopy methods for nonconvex optimization~\cite{mobahi2015link}, see also~\cite{simoes2021lasry,heaton2023global}. It is, however, not clear if regularity and invexity, perhaps after adapting~\eqref{equ:vese}, are indeed preserved along a solution $u(t,x)$. We believe this is interesting future work.  
}

\section{Appendix}
In this appendix we present auxiliary results on topology and switched systems. 
\label{sec:app}
{
\subsection{Homotopies and path-connectedness}
\label{sec:homotopy:path}
We frequently refer to spaces of stable dynamical systems and ask if such a space is path-connected or not, however, without making precise how to think of continuous curves in such a space. In this appendix, we briefly highlight how to go about this. We recall that we identify continuous vector fields on $\mathbb{R}^n$ with elements of $C^0(\mathbb{R}^n;\mathbb{R}^n)$. We will address in which sense the homotopy $H:[0,1]\times \mathbb{R}^n\to \mathbb{R}^n$ provides a continuous path in the space $C^0(\mathbb{R}^n;\mathbb{R}^n)$.
Due to space constraints, the discussion is brief, but for more details, we refer the reader to~\cite{ref:hirsch1976differential,munkres2014pearson}.

First, a topological space $X$ is said to be \textit{path-connected} if for any two points $x_0,x_1\in X$ there exists a continuous map $\gamma:[0,1]\to X$, a path, such that $\gamma(0)=x_0$ and $\gamma(1)=x_1$. Then, to reason about continuous curves in $C^0(\mathbb{R}^n;\mathbb{R}^n)$ we need to endow it with a topology, that is, we need to decide when two continuous maps are ``\textit{close}''. Leaving $\mathbb{R}^n$ for the moment, given two topological spaces $X$ and $Y$, then, for $K$ a compact subset of $X$ and $U$ an open subset of $Y$, sets of the form $\mathcal{O}(K,U):=\{f:f\in C^0(X;Y),\, f(K)\subset U\}$ comprise a \textit{subbasis} for the \textit{compact-open topology} on $C^0(X;Y)$. It turns out that this is the appropriate topology, as one can show the following. Let $X,Y$ and $Z$ be topological spaces with $X$ locally compact Hausdorff and endow $C^0(X;Y)$ with the compact-open topology, then, the map $H:Z\times X\to Y$ is continuous if and only if the map $h:Z\to C^0(X;Y)$ is continuous, where $h$ is defined by $(h(z))(x)=H(z,x)$~\cite[Thm.~46.11]{munkres2014pearson}. In particular, pick $Z=[0,1]$ and $X=Y=\mathbb{R}^n$, then, the existence of a homotopy $H:[0,1]\times \mathbb{R}^n\to \mathbb{R}^n$ is equivalent to a continuous path in $C^0(\mathbb{R}^n;\mathbb{R}^n)$.
}

{
\subsection{Switched systems}
Suppose we have a finite set of locally Lipschitz vector fields $\mathcal{F}=\{F_1,\dots,F_n\}$ on $\mathbb{R}^n$ such that the origin $0\in \mathbb{R}^n$ is GAS under any $F_i\in \mathcal{F}$. Now one might be interested in understanding if $0$ is still GAS under arbitrary switching between elements of $\mathcal{F}$, that is, to understand if $0$ is GAS under the \textit{switched system}
\begin{equation}
\label{equ:switched}
    \frac{\mathrm{d}}{\mathrm{d}t}x(t) = F_{\sigma(t)}(x(t)),
\end{equation}
where $t\mapsto \sigma(t)$ is a piecewise constant function taking values in $\{1,\dots,n\}$. It is known that for this to be true a common $C^{\infty}$ Lyapunov function $V$ must exist~\cite{ref:mancilla2000converse}. However, by the proceeding arguments we know that this implies that any $F_i\in \mathcal{F}$ can be homotoped to $-\nabla V$ such that $0$ remains GAS along the homotopy. Then, by the transitive properties of homotopies, this implies that for any two elements of $\mathcal{F}$ there must be a homotopy between them such that along the homotopy $0$ remains GAS. Hence, a somewhat counterintuitive statement is the following.

\begin{proposition}[Necessary condition for switched stability]
\label{prop:nec:switched}
    The origin $0\in \mathbb{R}^n$ is GAS under~\eqref{equ:switched} only if all elements of $\mathcal{F}$ belong to the same path-connected component of the space of continuous vector fields on $\mathbb{R}^n$ for which $0$ is GAS. 
\end{proposition}

Hence, Proposition~\ref{prop:nec:switched} further motivates studying the topology of the space of vector fields with a common attractor,~\textit{e.g.}, vector fields on $\mathbb{R}^n$ such that $0$ is GAS. {It is interesting to note that under the aforementioned conditions, the switched system~\eqref{equ:switched} can be continuously transformed to a negative gradient flow, without sacrificing stability. To that end, simply construct the maps $H_{\sigma}:[0,1]\times \mathbb{R}^n\to \mathbb{R}^n$ defined by $H_{\sigma}(s,x)=(1-s)F_{\sigma}(x)-s\nabla V(x)$ and observe that for any $s\in [0,1]$ the origin is GAS under
\begin{equation*}
    \frac{\mathrm{d}}{\mathrm{d}t}x(t) = H_{\sigma(t)}(s,x(t)). 
\end{equation*}
}

For more on switched systems we refer the reader to~\cite{ref:liberzon2003switching}. In particular, we point the reader to~\cite[Rem.~2.1]{ref:liberzon2003switching} for subtleties with respect to Lyapunov functions for switched systems. {In fact, from the same point of view, one observes that a necessary condition for $0$ to be GAS under~\eqref{equ:switched} is that $0$ is GAS under any element of the convex hull of $\mathcal{F}$,~\textit{e.g.}, consider $\langle \theta F_i + (1-\theta)F_j,\nabla V \rangle$ for $\theta\in [0,1]$~\cite[Cor. 2.3]{ref:liberzon2003switching}. }

\printbibliography[title={References}]



\end{document}